\documentclass{amsart}

\usepackage[T1]{fontenc}
\usepackage{lmodern}
\usepackage{comment}

\usepackage[francais,english]{babel}

\usepackage{latexsym}
\usepackage{amssymb}

\usepackage[centering,textheight=50.5pc,textwidth=30pc]{geometry}
\usepackage[pdfstartview=XYZ]{hyperref}
\theoremstyle{plain}
\newtheorem{theorem}{Theorem}[section]
\newtheorem{proposition}[theorem]{Proposition}
\newtheorem{lemma}[theorem]{Lemma}
\newtheorem{corollary}[theorem]{Corollary}

\theoremstyle{definition}
\newtheorem{definition}[theorem]{Definition}

\theoremstyle{remark}
\newtheorem{remark}[theorem]{Remark}

\theoremstyle{conjecture}
\newtheorem{conjecture}[theorem]{Conjecture}

\def\Z{\mathbb{Z}}

\title[Orders of primes and a conjecture of Ryser]{Quadratic residues and a new infinity of orders for which a conjecture of 
Ryser about  Circulant Hadamard matrices holds}

\date{\today}

\author[L. H. Gallardo]{Luis H. Gallardo}

\address{University Of Brest,
Mathematics\\
6, Av. Le Gorgeu\\
C.S. 93837\\
29238 Brest Cedex 3, France.}

\email{Luis.Gallardo@univ-brest.fr}



\begin{document}



\begin{abstract}
For every positive integer $k$ such that $k>1,$ there are an infinity of odd integers $h$
with  $\omega(h) =k$ distinct prime divisors such that there
do not exist a Circulant Hadamard matrix $H$ of order $n=4h^2.$
Moreover, our main result implies that for
all of the odd numbers $h$, with $1< h < 10^{13}$ there is no Circulant Hadamard matrix of order $n=4h^2.$ 
\end{abstract}

\maketitle


\section{Introduction}

A complex matrix $H$ of order $n$ is \emph{complex Hadamard} if $HH^{*} = nI$, where $I_n$ is the identity matrix of order $n$,
and if every entry  of $H/\sqrt{n}$ 
is in the complex unit circle. Here, the $*$ means transpose and conjugate. 
When such $H$ has real entries, so that $H$ is a $\{-1,1\}-$matrix, $H$ is called \emph{Hadamard}.  If $H$ is Hadamard
and circulant, say $H=circ(h_1,\ldots,h_n)$, that means that the $i$-th row $H_i$ of $H$ is given by 
$H_i = [h_{1-i+1},\ldots, h_{n-i+1}]$, the subscripts being taken modulo $n,$ for example
$H_2 =[h_n,h_1,h_2, \ldots,h_{n-1}].$ A long standing conjecture of Ryser (see \cite[pp. 134]{ryser})
is:

\begin{conjecture}
\label{mainryser}
Let $n \geq 4.$  If $H$ is a circulant Hadamard matrix of order $n$, then $n=4.$
\end{conjecture}

Details about previous results on the conjecture and a short sample of recent related papers are in
\cite{turyn1}, \cite{Leung}, \cite{BorMoss1}  \cite{BorMoss2}, \cite{Craigen},
\cite{brualdi}, \cite{EGR} and the bibliography therein, \cite{lhg}.

The object of the present paper is to substantially extend the range of known $n$'s for which Ryser's Conjecture holds.

Our main result is:

\begin{theorem}
\label{mainar}
For every positive integer $k$ such that $k>1,$ there are an infinity of odd integers $h$
with  $\omega(h) =k$ distinct prime divisors such that there
do not exist a circulant Hadamard matrix $H$ of order $n=4h^2.$
\end{theorem}

Our result is a simple consequence (see Lemma \ref{maincomp}) of a deep result of Arasu 
(see \cite[Theorem 4, part (i)]{Arasu} and Lemma \ref{arasumain} below).   
By using the Hadamard-Barker 's data in the web site of M. Mossinghoff (see \cite{webmoss}) and our
main key Lemma  \ref{maincomp} below, we are also able to prove (by computer computations) the following new
result.

\begin{proposition}
\label{computA}
Let $S$ be the set of all integers $n=4h^2$ with odd $h$
such that $1<h < 10^{13}.$
For all $s \in S$ there is no circulant Hadamard matrices of order $n.$
\end{proposition}

These results implies corresponding results for the existence of Barker sequences (see section \ref{barkers}).
Of course, by using Lemma \ref{maincomp} combined with results obtained by other methods it is possible
to improve our numerical results herein.  For example, by applying the lemma to already known $h$'s satisfying $h < 10^{24}$
(see some of them in the web site already cited) and for which all other methods have failed, etc.  Two more examples: (a)
In about $6$ seconds computation in an old computer our Lemma \ref{maincomp} eliminated the only  possible obstruction known for $h$, namely
$h = 31540455528264605$ in order that a Barker sequence exist with $13 < 4h^2 < 10^{33}$, (see \cite[Theorem 1]{BorMoss2}).
(b) In about $3$ seconds the first $6$ values of $h$ in between $10^{16.5}$ and $5\cdot 10^{24}$
$$ 
[66687671978077825,866939735715011725,1293740836374709805,
$$
$$6468704181873549025,
16818630872871227465,
84093154364356137325];
$$
(over $18$) in  \cite[Table $2$]{BorMoss2} were also eliminated
as before. However, it is easy to see that there are values of $h$ that satisfy all the assumptions of Lemma  \ref{maincomp},
besides the assumption on the possible existence of $H$.
Indeed some experiments with small values of $h$, say $h \leq 10000,$ suggest that, at least for these values,
there exist about $5/100$ of $h$'s for which
all the orders appearing in the lemma are odd.

\section{Some tools}

First of all we recall the notion of a weighing matrix.

\begin{definition}
\label{weig}
Let $n$ be a positive integer.  Let $k$ be a positive integer. A \emph{weighing} matrix $W$ of order $n$
and weight $k$ is an $n \times n$ matrix $W$ with all its entries belonging to the set $\{-1,0,1\}$ such that
$$
W W^{T} = k I_n
$$
where the ``$T$" means ``transpose'' and  $I_n$ is the identity matrix of order $n.$
\end{definition}

We recall the result of Arasu (\cite[Part (i) of Theorem 4]{Arasu}).

\begin{lemma}
\label{arasumain}
Let $n,k$ be positive integers such that $n=p^a\cdot m$, $k = p^{2b} \cdot u^2$, where $a,b,m,u$ are
positive integers such that the prime number $p$ does not divide $m$
and $p$ does not divide $u.$
Assume that there exists an integer $t$ such that
$$
p^t \equiv -1 \pmod{m}.
$$
If there exists a weighing matrix $W$ of order $n$ and of weight $k$ that is circulant then
$p=2$ and $b=1.$
\end{lemma}

We use the obvious decomposition below of a circulant Hadamard matrix of  even order $n$ in four blocks of order $n/2,$
(see \cite{lhg} for another result based on the same decomposition), in order to build a weighing matrix attached to $H$.

\begin{lemma}
\label{weighing}
Let $H = circ(h_1, \ldots,h_n)$ be a circulant Hadamard matrix of order $n.$ Then
\begin{itemize}
\item[(a)] 
\begin{equation*}
H=
\begin{bmatrix}
A & B\\
B & A\\
\end{bmatrix}
\end{equation*}
where $A,B$ are matrices of order $\frac{n}{2}.$
\item[(b)]
$K=A+B$ is circulant with entries in $\{-2,0,2\}.$
\end{itemize}
\end{lemma}

We build now the weighing matrix.

\begin{lemma}
\label{weighh}
Let $h$ be an odd positive integer. Assume that $H$ is a circulant Hadamard matrix of 
order $n,$ where $n=4h^2.$ Then, there exists a weighing matrix  $C$ of order $n/2$ and weight $n/4.$
\end{lemma}

\begin{proof}
Set $C = \frac{A+B}{2}$ where $A$ and $B$ are defined by Lemma \ref{weighing}. One has then that $C$
is circulant, of order $n/2 = 2h^2$ with all its entries in $\{-1,0,1\}.$  From $H H^{*} =I_n,$ one gets by block multiplication
$AA^{*}+BB^{*} = n I_{n/2}$  and  $AB^{*}+BA^{*} =0.$  Thus,
\begin{equation}
\label{weic}
4 \cdot CC^{*}= AA^{*}+AB^{*}+BA^{*}+BB^{*} = AA^{*}+BB^{*} =n I_{n/2}.
\end{equation}
It follows from \eqref{weic} that $C$ is a weighing circulant matrix of order $n/2$ and weight $n/4.$
\end{proof}

We are now ready to show our main result from which, (essentially), we will be obtaining all our results.

\begin{lemma}
\label{maincomp}
Let $h$ be an odd positive integer exceeding $1$. Assume that $H$ is a circulant Hadamard matrix of 
order $n,$ where $n=4h^2.$ Let $p$ be a prime divisor of $h$ and  $r$ be the positive integer such that
$p^r \mid h$ but $p^{r+1} \nmid h.$ Set $s = h/p^r.$  Let $o_{m}(p)$ be the order of $p$ in the multiplicative group
$G = (\Z/m\Z)^{*}$
of inversible elements of the ring $\Z/m\Z$, where $m=2s^2.$ Then,
$$
o_m(p) 
$$
is an odd number.
\end{lemma}

\begin{proof}
Assume, to the contrary, that $o_{2s^2}(p)$ is even, say $o_{2s^2}(p) =2f.$ Then $p^f \equiv -1 \pmod{2s^2}.$
Then, by Lemma \ref{arasumain} applied to the weighing matrix $C$, of order $n/2$ and weight $n/4$,
defined by Lemma \ref{weighh} 
with $a = 2r,$ $b =r,$  $m = 2s^2$, and $u=s$ that are all positive integers, and observing that we have
$\gcd(p,m)=1$ and $\gcd(p,u)=1$, we obtain the contradiction
\begin{equation}
\label{tiersexclu}
p=2.
\end{equation}
This proves the lemma.
\end{proof}

\begin{remark}
\label{obs}
Of course, if in the proof of Lemma \ref{maincomp},
we apply Lemma \ref{arasumain}  to the full circulant weighing matrix $H$ (of order $n$, and of weight $n$),
instead to applying it to $C$,
we obtain no contradictions.
\end{remark}

In order to complete the results the following simple arithmetic result is key.

\begin{lemma}
\label{gauss}
Let $p$ and $q$ be two odd prime numbers such that the orders $o_q(p),$ the order of $p$ modulo $q,$
and $o_p(q),$ the order of $q$ modulo $p,$ are both odd. Then
$$
{p \overwithdelims () q} =1 \;\;\;\text{and}\;\;\; {q \overwithdelims () p} =1.
$$
where $\cdot \overwithdelims () \cdot$ is the Legendre's symbol.
\end{lemma}

\begin{proof}
Since $d := o_q(p)$ is odd, we have $p \equiv {\left ({(1/p)}^{\frac{d-1}{2}} \right )}^2 \pmod q.$ Analogously
$e := o_p(q)$ odd implies $q \equiv {\left ({(1/q)}^{\frac{e-1}{2}}\right)}^2 \pmod p.$  The result follows.
\end{proof}

\section{ Proof of Theorem \ref{mainar} and of Proposition \ref{computA}}

\subsection{Proof of Theorem \ref{mainar}}

Assume that there are only a finite number of such odd integers $h.$  Then, by Lemma \ref{maincomp},
there exists some odd positive integer
$h_0$ such that for any odd integer $h$ with $h \geq h_0$  every prime number $p$ such that $p \mid h$,
say, $h = p^r \cdot s$, with $p^{r+1} \nmid h,$ satisfy
\begin{equation}
\label{oor}
o_{2s^2}(p)\;\;\text{is odd}.
\end{equation}
Thus, \eqref{oor} implies that for every odd prime divisor $q$ of $2s^2$ one has
\begin{equation}
\label{oor1}
o_{q}(p)\;\;\text{is odd}.
\end{equation}
Write now $h = q^{t}d$ with $q \nmid d.$ one has also,
\begin{equation}
\label{oor2}
o_{2d^2}(q)\;\;\text{is odd}.
\end{equation}
Thus, for every odd prime divisor $r$ of $2d^2$ one has
\begin{equation}
\label{oor3}
o_{r}(q)\;\;\text{is odd}.
\end{equation}
Choose $r=p$ in \eqref{oor3}.  One gets
\begin{equation}
\label{oor4}
o_{p}(q)\;\;\text{is odd}.
\end{equation}

This implies, by Lemma \ref{gauss} that  ${p \overwithdelims () q} =1$ and that 
${q \overwithdelims () p} =1$ for any other prime factor $q$ of $h.$  But this is false, since we can always choose
two distinct primes $p_1$ and $p_2$ both larger than $h_0$ and with, e.g.,
$$
{p_1 \overwithdelims () p_2} =1 \;\;\;\text{and}\;\;\; {p_2 \overwithdelims () p_1} =-1,
$$
and take
$$
h = p_1 \cdot p_2 \cdots p_k
$$
with any other distinct prime numbers (when $k>2$), $p_2, \ldots, p_k.$
This proves the theorem.

\subsection{Proof of Proposition \ref{computA}}

It is known (see \cite{BorMoss2}) that for all elements of $S$ but for a subset $T$
containing $1371$ elements $h$ the result holds. Using Lemma \ref{maincomp} and a
straightforward (included below for completeness) computer program that checked the conclusion
of the above lemma for all these $h$'s, and after about $7$ minutes of computation,
we obtained the result.

Here the program used:

\begin{verbatim}
#  n's 4*h**2, with constraints on its odd prime divisors

with(numtheory):

tesp := proc(h)
local p,m,par,pris,el,mo,rr;
pris := ifactors(h); pris := op(2,pris);
if nops(pris) = 1 then RETURN(0); fi;
for par in pris do
p := op(1,par); m := op(2,par); el := iquo(h,p**m); mo := 2*el**2; 
rr := order(p,mo);
if modp(rr,2) = 0 then RETURN(0); fi;
od;
RETURN(1);
end;

#  checks the 1371 elements of the list uvals

seelm := proc()
local p,lis,c,st;
st := time(); c := 0; lis :=[];
read(mike1):
for p in uvals do
if tesp(p) = 1 then print([c,[1371],p]); lis := [p,op(lis)]; fi;
c := c+1; 
if modp(c,100) = 0 then print([time() -st,c]) fi;
od;
lis;
end;

# the actual program runned is:

interface(prettyprint=0):
interface(quiet=true): st := time(); time() -st;
st := time(); z := seelm(); time() -st;

quit;
\end{verbatim}

\section{Barker sequences}
\label{barkers}

Suppose $x_1,x_2, \ldots, x_n$ is a sequence of $1$ and $-1$. We recall the following definition.

\begin{definition}
\label{barkerD}
A sequence $c_1,c_2, \ldots, c_{n-1},$ where
$$
c_j = \sum_{i=1}^{n-j} x_i \cdot  x_{i+j}
$$
and the subscripts are defined modulo $n$, is called a \emph{Barker} sequence of length $n$ provided
$c_j \in \{-1,0,1\},$ for all $j=1,2, \ldots n-1.$
\end{definition}

The main known result is the following, (see  \cite{turyn1}, \cite{shalomK}).

\begin{lemma}
\label{barkerk}
If there exists a Barker sequence of length $n>13$ then there exists a circulant Hadamard matrix of order $n.$
\end{lemma}

\begin{corollary}
\label{barkerdone}
For an infinity of odd integers $h$'s with an arbitrary fixed  number $k$ of  distinct prime divisors
 there do not exists a Barker sequence of length $4h^2 >13.$ Moreover,  there do not exists
a Barker sequence of length $4h^2 >13$ for all odd integers $h$ such that
$1<h<10^{13}.$
\end{corollary}

\begin{proof}
Follows from Lemma \ref{barkerk}, from Theorem \ref{mainar} and from Proposition \ref{computA}.
\end{proof}

\section*{Acknowledgements}

We are indebted to Michael J. Mossinghoff for sending us a special file, ready for computations,
containing all odd integers $h$ less than $10^{13}$ (used in the proof of Proposition \ref{computA}, above)
for which all previous known results, cannot decide whether or not for $n=4h^2$ there exist a circulant  Hadamard matrix
of order $n.$ Thanks also to Carlos  M. da Fonseca for patience and to the anonymous referees of a related paper
of mine, containing a too optimistic result, for inspiration to build the present paper.

\end{document}